\documentclass[12pt]{amsart}

\usepackage{amssymb}

\usepackage{enumitem}

\usepackage{graphicx}
\usepackage[all]{xy}
\usepackage[pdftex,bookmarks,colorlinks,breaklinks]{hyperref}
\makeatletter
\@namedef{subjclassname@2020}{%
  \textup{2020} Mathematics Subject Classification}
\makeatother

\newtheorem{theorem}{Theorem}[section]
\newtheorem{corollary}[theorem]{Corollary}
\newtheorem{lemma}[theorem]{Lemma}
\newtheorem{proposition}[theorem]{Proposition}

\newtheorem{fact}[theorem]{Fact}

\newtheorem{claim}[theorem]{Claim}
\theoremstyle{definition}

\newtheorem{remark}[theorem]{Remark}
\newtheorem{example}[theorem]{Example}
\newtheorem{conjecture}[theorem]{Conjecture}

\newtheorem*{thma}{Theorem A}
\newtheorem*{thmb}{Theorem B}
\newtheorem*{thmc}{Theorem C}

\newcommand{\T}{\mathbb{T}}
\newcommand{\Z}{\mathbb{Z}}

\newcommand{\R}{\mathbb{R}}

\newcommand{\Q}{\mathbb{Q}}
\def\nbd{neighbourhood}
\def\hull#1{\langle#1\rangle}

\numberwithin{equation}{section}

\usepackage{color}
\newenvironment{revvy}{\color{magenta}}{}
\newenvironment{revfl}{\color{blue}}{}

\def\by{\begin{revvy}}
\def\ey{\end{revvy}}
\def\bfl{\begin{revfl}}
\def\efl{\end{revfl}}

\begin{document}

%%%%% To ease editing, for IMPAN journals add:

\baselineskip=17pt

%%%%%%%%%%%%%%%%

\title[Groups with Dense Subgroups Separable]{Locally Compact Groups with All Dense Subgroups Separable}
\author[D. Peng]{Dekui Peng}\thanks{This work is supported by NSFC grants NO. 12271258 and 12301089, and the Natural Science Foundation of the Jiangsu Higher Education
Institutions of China (Grant NO. 23KJB110017).}
\email{pengdk10@lzu.edu.cn}
\address{Institute of Mathematics, Nanjing Normal University, Nanjing 210046, China}

%%%%%%%%%%%%%%
\begin{abstract}
By a recent result of Juh\'{a}sz and van Mill, a locally compact topological group whose dense subspaces are all separable is metrizable. In this note we investigate the following question: is every locally compact group having all dense subgroups separable also metrizable?
We give an example to show the answer is negative for locally compact abelian groups, thereby showing that one cannot directly generalize the assertion by replacing ``subspaces'' with ``subgroups''. On the other hand, we prove that the answer is positive for compact groups which are either connected or algebraically abelian; and for locally compact groups containing only separable subgroups. As an application, we obtain a necessary condition for metrizability of pronilpotent groups.

\end{abstract}

\maketitle

%We denote by $w(X)$ the weight of $X$, then $d(X)\leq w(X)$.

\section{Introduction}\label{Sec1}
 Let $X$ be a Hausdorff topological space.
The {\em density} $d(X)$
 of $X$ is the minimum cardinal $\tau$ such that $X$ admits a dense subset of size $\tau$. The space $X$ is called {\em separable} if $d(X)$ is countable.
Let $dd(X)$ denote the \emph{double density spectrum} of $X$, which is defined as $$dd(X)=\{d(Y):\mbox{$Y$~is~a~dense~subspace~of~$X$}\}.$$
This notion was introduced in a recent article \cite{JvMSS} of I. Juh\'{a}sz, J. van Mill, L. Soukup, and Z. Szentmikl\'{o}ssy, where the double density spectra of compact spaces were studied (see \cite[Section 3]{JvMSS}). In the same research line,
Juh\'{a}sz and van Mill \cite{JvM}
proved that the double density spectrum $dd(G)$
coincides with the interval $[d(G), w(G)]$ for any locally compact group $G$, where $w(G)$ is the weight of $G$. In particular, this implies that $w(G)=d(G)$ holds whenever $dd(G)=\{d(G)\}$ holds for a locally compact group $G$. This observation leads to the following consequence:

\begin{fact}\label{Prop1}If $G$ is a locally compact group with all dense subspaces separable, then $G$ is metrizable.\end{fact}

For the category of topological groups, a natural question would be to determine if one can replace the word ``subspaces'' by ``subgroups'' in the above result. Namely: is any locally compact group with all dense subgroups separable necessarily a metrizable group?

We shall give a negative answer to this question by presenting two examples:
there exist a non-metrizable locally compact abelian group $G$ and a non-metrizable connected locally compact  group $H$ such that all dense subgroups of $G$ and $H$ are separable (see Examples \ref{this:example:feb62024} and \ref{this:example:feb82024}).

%\by We also give a positive answer for two elementary types of compact groups:\ey
With the goal of finding positive answers, we develop two combinatorial results for compact groups in this note:

\begin{thma}{\em Every compact abelian group $G$ admits a dense subgroup $H$ such that $d(H)=w(G)$.}\end{thma}

The question of whether a
 compact abelian group with all dense subgroups separable is metrizable was also considered by F. Lin, Q. Wu and C. Liu \cite{LWL}. 
They proved that a compact abelian group with only separable dense subgroups is metrizable if it satisfies one of the following: (1) torsion; (2) torsion-free; (3) connected (see \cite[Theorems 3.5, 3.6 and 3.8]{LWL}). Of course, all the results can be covered by our Theorem A.
We note that our proof of Theorem A avoids these results. The main tool for our argument is the classical duality theorem for precompact abelian groups of W. Comfort and K. Ross \cite{CR1}.  %and further developed by many mathematicians, including the author (see \cite{Peng}).
 The proof of Theorem A is given in Section \ref{proof:ofthmA}.

Our second result involves connectedness:
\begin{thmb}{\em Every connected compact group $G$ admits a dense subgroup $H$ such that $d(H)=w(G)$.}\end{thmb}

 The proof of Theorem B is reserved for Section \ref{proof:ofthmB}. By combining
Theorems A and B
one gets the following:

\begin{corollary}Let $G$ be a compact group with all dense subgroups separable. If $G$ is abelian or connected, then $G$ is metrizable.\end{corollary}

We provide two examples showing that Theorems A and B do not hold in the realm of locally compact
 groups (see Examples \ref{this:example:feb62024} and \ref{this:example:feb82024}).
%Examples \ref{this:example:feb62024} and \ref{this:example:feb82024} show that Theorems A and B do not hold in the real of locally compact topological groups.
 However, we are still able to apply Theorem B 
in the case of a locally compact group $G$ satisfying $w(G) > 2^\omega$ (see Proposition \ref{this:prop:feb82024}). This result is achieved through a delicate application of the Local Splitting Theorem of Iwasawa \cite{Iwa}.

We suspect the following stronger statement to be true.
\begin{conjecture}Every compact group $G$ has a dense subgroup $H$ such that $d(H)=w(G)$.\end{conjecture}

The difficulty to prove the above conjecture is to find a suitable dense subgroup. However, things may become much easier if one drops the restriction of denseness.

Our last main result in this note is the following:

\begin{thmc}{\em Let $G$ be an infinite locally compact group. Then $w(G)$ is the supremum of the densities of subgroups of $G$.
If the cardinal $w(G)$ has uncountable cofinality, then $G$ has a subgroup satisfying $d(H)=w(G)$. }\end{thmc}

The proof of Theorem C will be given in Section \ref{proof:ofthmC}. As a consequence of it, we deduce the following:

\begin{corollary}
A locally compact group containing only separable subgroups is metrizable.
\end{corollary}

To close this paper, we give an application of Theorem A in the realm of the transfinite nilpotent profinite groups in Section \ref{transfinite:section}.

\subsection{Notation and Terminology}

Hereafter, all
topological spaces are assumed to be Hausdorff, unless otherwise stated. For a subset $Y$ of a topological space $X$, $\overline{Y}$ is the closure of $Y$.
Abelian groups are
written additively, including the torus $\T$, so the neutral element of an abelian group will be denoted by $0$.
For an integer $n$ and an abelian group $G$, $nG$ is the subgroup $\{ng:g\in G\}$.
Let $X$ be a subset of a group $G$, the subgroup of $G$ generated by $X$ is denoted by $\hull{X}$.
If $X=\{g\}$, we shall write simply $\hull{g}$ instead of $\hull{\{g\}}$.
We denote by $\Z$ the additive group of integers, by $\Q$ the rationals and by $\R$ the reals.
For any positive integer $n$, $\Z(n)$ is the cyclic group with $n$ elements; $\Z(p^\infty)$ is the Pr\"{u}fer group for a prime $p$.
As usual, $H\leq G$ means that $H$ is a subgroup of $G$. Let $H, K\leq G$, we denote by $[H,K]$ the commutator subgroup, i.e., the subgroup of $G$ generated by $\{hkh^{-1}k^{-1}: h\in H, k\in K\}$.

Let $\{G_\alpha:\alpha<\tau\}$ be a family of topological groups, and denote the identity of each group $G_\alpha$ by $1$.
For any element $g=(g_\alpha)_{\alpha<\tau}\in \prod_{\alpha<\tau} G_\alpha$, the {\em support} of $g$ is the subset $supp(g)=\{\alpha<\tau: g_\alpha\neq 1\}$ of $\tau$.
The set of all elements in $\prod_{\alpha<\tau} G_\alpha$ with finite support is a (dense) subgroup and called the {\em restricted product} or the {\em $\sigma$-product} of $\{G_\alpha: \alpha<\tau\}$, which is also known as the {\em direct sum} (most usually in the abelian case).
This group shall be denoted by $\sigma\prod_{\alpha<\tau}G_\alpha$.
For a group $G$ and a cardinal $\tau$, we shall write $G^{(\tau)}$ for the restricted product of $\tau$ many copies of $G$.

For a locally compact abelian group $G$, the Pontryagin-van Kampen dual group is denoted by $\widehat{G}$, i.e., the group of all continuous homomorphisms of $G$ to $\T$ endowed with the compact-open topology.
A topological group is called {\em precompact} if it is a dense subgroup of a compact group.
Now let $G$ be an abstract (discrete) abelian group.
Then the Comfort-Ross Duality \cite{CR1} yields that every Hausdorff precompact group topology $\sigma$ on $G$ is generated by a dense subgroup $D$ of $\widehat{G}$, i.e., $\sigma$ is the coarsest topology making all $\chi\in D$ continuous.
In this case, $D$ is called the Comfort-Ross dual group of $(G, \sigma)$ and denoted by $(G, \sigma)^*$.
So the underlying group of $(G, \sigma)^*$ is the group of all continuous homomorphisms from $(G, \sigma)$ to $\T$ and the topology of $(G, \sigma)^*$ is the subspace topology inherited from $\widehat{G}$, which coincides with the topology of pointwise convergence.
If $f: G\to H$ is a continuous homomorphism of precompact abelian groups, then there exists a natural continuous homomorphism $f^*:H^*\to G^*, \chi \mapsto \chi\circ f$ (see \cite{Peng}).
The following elementary facts about the Comfort-Ross Duality, which can be found in \cite{Peng}, will be used in this note.

\begin{fact}\label{fac}Let $G$ and $H$ be precompact abelian groups. Then
\begin{itemize}
  \item[(a)] $G\cong (G^*)^*$;
  \item[(b)] $w(G^*)=|G|$ and $w(G)=|G^*|$;
  \item[(c)] an onto continuous homomorphism $f: G\to H$ is a bijection if and only if $f^*:H^*\to G^*$ is a topological group embedding with dense image. Conversely, $f:G\to H$ is a topological monomorphism with $f(G)$ dense in $H$ if and only if $f^*:H^*\to G^*$ is a continuous one-to-one onto homomorphism.
\end{itemize}
\end{fact}

Throughout this note, $\omega$ will be the first infinite (countable) cardinal.
For a set $X$, $|X|$ is the cardinality of $X$. If $X$ is a topological space, $d(X)$ is the density of $X$ and $w(X)$ is its weight.
%If $G$ is a topological group and $N$ is a closed subgroup, then it is well-known and easily verified that $w(G)=w(N)\cdot w(G/N)$, where $G/N$ is the left coset space.
The  {\em character} of a topological group $G$ is the smallest cardinal of a local base of the identity. We shall denote this cardinal function by $\chi(G)$\footnote{In many references, the term ``character'' is used to denote a homomorphism of an abelian group to the torus $\T$.
For any topological group $G$, the equality $w(G)=d(G)\cdot \chi(G)$ holds (see \cite[Lemma 5.1.7]{ABD} or \cite[Theorem 5.2.5]{AT}).
The following results are well-known but challenging to locate a reference for in academic literature.}

\begin{lemma}\label{3space}
Let $G$ be a topological group and $N$ a closed subgroup. Then $w(G)=w(N)\cdot w(G/N)$, where $G/N$ is the left coset space.
\end{lemma}
\begin{proof}
When $G$ is finite, this is trivial. Suppose now $G$ is infinite.
The direction $w(G)\geq w(N)\cdot w(G/N)$ is evident. 
According to  \cite[5.38(f)]{HR}, we know that $$d(G)\leq d(N)\cdot d(G/N)\leq w(N)\cdot w(G/N).$$
Thus, if we can prove that $\chi(G)\leq \chi(N)\cdot \chi(G/N)$, then by $\chi(N)\cdot \chi(G/N)\leq w(N)\cdot w(G/N)$ and the equality mentioned above we would have that
$w(G)=d(G)\cdot \chi(G)\leq w(N)\cdot w(G/N).$

The following is a modification of the proof of \cite[Theorem 1.5.20]{AT}, in where the desired inequality for the case $\chi(N)=\chi(G/N)=\omega$ is proved.

Let $\mathcal{B}_1$ be a family of symmetric\footnote{a subset $A$ in a group is called {\em symmetric} if $A=A^{-1}$}  identity \nbd s in $G$ such that (1) $|\mathcal{B}_1|=\chi(N)$; (2) for any $V\in \mathcal{B}_1$, there exists $V'\in \mathcal{B}_1$ with $V'V'\subseteq V$; and (3) $\{V\cap N: V\in \mathcal{B}_1\}$ is a local base at the identity in $N$.
Fix another family $\mathcal{B}_2$ of open \nbd s of the identity in $G$ of cardinality coinciding with $w(G/N)$ such that $\{WN/N: W\in \mathcal{B}_2\}$ is a local base of the point $N$ in the quotient space $G/N$.

We will see that $\mathcal{B}:=\{V\cap W: V\in \mathcal{B}_1, W\in \mathcal{B}_2\}$ is a local base at the identity in $G$, thus ensuring that $\chi(G)\leq \chi(N)\cdot \chi(G/N)$.
For any identity \nbd~ $U$ of $G$, take an identity \nbd~ $U'$ with $U'U'\subseteq U$.
Let $V, V'\in \mathcal{B}_1$ such that $V\cap N\subseteq U'$ and $V'V'\subseteq V$.
Since $V'\cap U'$ is of course open in $G$, by the chosen of $\mathcal{B}_2$, we can find $W\in \mathcal{B}_2$ such that $W\subseteq (V'\cap U')N$.
Now assume that $x\in V'\cap W$. Then there exist $y\in V'\cap U'$ and $z\in N$ with $x=yz$.
So $z=y^{-1}x\in V'V'\subseteq V$.
As $z$ is also in $N$, we have $z\in V\cap N\subseteq U'$.
Hence we have $x=yz\in U'U'\subseteq U$, implying that $V'\cap W\subseteq U$.
This inclusion completes the proof.
\end{proof}
For an infinite cardinal number $\tau$, a subset of a topological space $X$ is called a {\em $G_\tau$-set} if it is the intersection of (at most) $\tau$ many open subsets of $X$.
In frequently used terminology, $G_\omega$-sets are known as $G_\delta$-sets.
The {\em pseudocharacter} of a topological group $G$, which is denoted by $\psi(G)$, is the least infinite cardinal $\tau$ such that the singleton $\{1_G\}$ is a $G_\tau$-set.

\section{two inequalities for pseudocharacters}

In this section we consider two inequalities for pseudocharacters that will be of use in the proof of Theorem A. The first lemma is a part of folklore, and we include its proof for convenience of the reader:

\begin{lemma}\label{Le0}Let $G$ be a precompact abelian group and $\tau$ an infinite cardinal.
Then, $d(G)\leq \tau$ holds if and only if $\{0\}$ is a $G_\tau$-set in $G^*$ (i.e., $\psi(G^*)\leq \tau$).\end{lemma}
\begin{proof}
First assume that $d(G)\leq \tau$.
Then, $G$ admits a dense subgroup $H$ of cardinality $\leq \tau$. We denote by $j: H\to G$ the corresponding identity embedding.
Fact \ref{fac}(b) yields that $w(H^*)\leq\tau$; so $\{0\}$ is a $G_\tau$-set in $H^*$.
Moreover, by Fact \ref{fac}(c), $j^*: G^*\to H^*$ is a monomorphism.
Hence, $\{0\}$ is also a $G_\tau$-set in $G^*$.

For the converse, assume $\{0\}$ is a $G_\tau$-set in $G^*$. Then, there exists a family $\{U_\alpha:\alpha<\tau\}$ of open sets in $G^*$ such that $\bigcap_{\alpha<\tau} U_\alpha=\{0\}$.
For each $\alpha<\tau$, one may take a finite number of continuous homomorphism $\chi_{\beta_1}, \chi_{\beta_2},...,\chi_{\beta_n}$ of $G^*$ to $\T$ such that
$$\bigcap_{i=1}^n\ker \chi_{\beta_i}\subseteq U_\alpha.$$
We then define the diagonal product $\chi_\alpha:=\Delta_{i=1}^n \chi_{\beta_i}$, which maps $G^*$ into $\T^n$ with the kernel contained in $U_\alpha$.
Hence, the diagonal product $\chi:=\Delta_{\alpha<\tau} \chi_\alpha: G^*\to \T^\tau$ is injective, implying that
$w(\chi(G^*))\leq \tau$.
Let $H=\chi(G^*)$ be the image of $G$ under $\chi$ as a topological subgroup of $\T^\tau$.
Then, $\chi^*: H^*\to (G^*)^*\cong G$ is a dense embedding by Fact \ref{fac}(c).
Item (b) of Fact \ref{fac} now implies that $\chi^*(H^*)$ is a dense subgroup of $G$ of cardinality less than or equal to $\tau$, that is, $d(G)\leq \tau$.
\end{proof}

%The main idea of the proof is to use the following lemma. In the following proof,

 For the next lemma,
for every subset $X$ of a cardinal $\kappa$, we shall identify the subproduct $\prod_{\alpha\in X}D_\alpha$ with the subset of $\prod_{\alpha<\kappa}D_\alpha$ consisting of all those elements $g$ with $supp(g)\subseteq X$, where each $D_\alpha$ is a group.
In particular, each $D_\alpha$ is a natural subgroup of $\prod_{\alpha<\kappa}D_\alpha$.
\begin{lemma}\label{Le1}Let $G$ be an infinite abelian group. Then $G$ admits a Hausdorff precompact group topology $\mathcal{T}$ such that $\psi(G, \mathcal{T})=|G|$.\end{lemma}
\begin{proof}
When $G$ is countable there is nothing need to check. So we assume that $G$ is uncountable.
Let $\kappa=|G|$ and $\tau$ be an infinite cardinal such that $\tau<\kappa$.

Denote by $D$ the {\em divisible hull} (see \cite[Section 24]{Fuc}) of $G$, which is a divisible abelian group containing $G$ as an essential subgroup, that is, $G\cap N\neq \{0\}$ for each $\{0\}\neq N\leq D$.
So one has $|D|=|G|=\kappa$.

By the structure theorem of divisible abelian groups \cite[Theorem 23.1]{Fuc}, $D$ splits into the direct sum of a family $\{D_\alpha:\alpha<\kappa\}$ of subgroups, where each $D_\alpha$ is isomorphic to either $\Q$ or $\Z(p^\infty)$ for some prime $p$.
In both cases, $D_\alpha$ embeds into the circle $\T$ as a dense subgroup.
Thus, one embeds $\Pi:=\prod_{\alpha<\kappa}D_\alpha$ into $\T^\kappa$.
Hence $D=\bigoplus_{\alpha<\kappa}D_\alpha$, which is isomorphic to the restricted product of $\{D_\alpha:\alpha<\kappa\}$, inherits a Hausdorff precompact group topology from $\T^\kappa$.

We claim that if $\mathcal{T}_D$ is the topology $D$ inherits from $\Pi$, then $\{0\}$ is not a $G_\tau$-set in $(D, \mathcal{T}_D)$.
Suppose, on the contrary, that $\{0\}$ is a $G_\tau$-set in $D$.
Then there exists a family $\{U_\alpha:\alpha< \tau\}$ of open subsets of $\Pi$ such that $\bigcap_{\alpha<\tau}U_\alpha\cap D=\{0\}$.
For each $\alpha< \tau$, there exists a finite subset $E_\alpha$ of $\kappa$ such that $\prod_{\beta\in \kappa\setminus E_\alpha}D_\beta\subseteq U_\tau$.
Then $E=\bigcup_{\alpha< \tau}E_\alpha$ is a subset of $\kappa$ of cardinality at most $\tau$ and $\prod_{\alpha\in \kappa\setminus E}D_\alpha\subseteq \bigcap_{\alpha<\tau}U_\alpha$.
Take $\beta\in \kappa\setminus E$. Then $D_\beta\subseteq \bigcap_{\alpha<\tau}U_\alpha\cap D=\{0\}$.
This is a contradiction.

Now let $P$ be any $G_\tau$-set of $D$ containing $0$.
Without loss of generality, one may assume that $P$ is a subgroup\footnote{One may assume  $P=\bigcap_{\alpha\in \tau}U_\alpha$ such that $U_{\alpha+1} U_{\alpha+1}\subseteq U_\alpha$ for all $\alpha<\tau$, where each $U_\alpha$ is a symmetric open neighbourhood of the identity. Then for every ordinal $\beta<\tau$, $\bigcap_{n\in\omega}U_{\beta+n}$ is a subgroup; therefore, $P$ is an intersection of subgroups.}.
By the above assertion, $P\neq \{0\}$.
The essentiality of $G$ in $D$ then yields that $P\cap G\neq \{0\}$.
Thus, $\{0\}$ cannot be a $G_\tau$-set in $G$.
So, with the subspace topology inherited from $D$, $G$ is a precompact abelian group with pseudocharacter equal to $|G|$.
\end{proof}

\section{Theorem A and its limitations} \label{proof:ofthmA}
%In this section, we provide a proof of Theorem A.

First, we present the proof of Theorem A.
\begin{proof}[\bf Proof of Theorem A]
Let $G$ be a compact abelian group and $K=\widehat{G}$ the (discrete) Pontryagin dual group of $G$.
According to Lemma \ref{Le1}, $K$ admits a Hausdorff precompact group topology such that the singleton $\{0\}$ is not a $G_\tau$-set in $(K, \mathcal{T})$, for each cardinal $\tau<w(G)$.
The Comfort-Ross dual group $H$ of $(K,\mathcal{T})$ is a dense
 subgroup of $G$ (by Fact \ref{fac}(c)), which has density $\geq w(G)$, by Lemma \ref{Le0}. The inequality $d(H)\leq w(H)$ is
 trivial since $d(H) \leq w(H) \leq w(G)$.
%So $d(H)\geq w(G)$.
%The direction $d(H)\leq w(G)$ is trivial since $d(H)\leq w(H)$ and $w(H)=w(G)$.
In summary, we have $d(H)=w(G)$.
\end{proof}

One may wonder whether Theorem A can be extended to a wider class of topological abelian groups than that of compact ones.
Two natural generalizations of compactness in topological groups are precompactness and local compactness.
However, the following examples show that Theorem A admits a generalization to neither of these two classes.

\begin{example}%We have the following examples:
\label{this:example:feb62024}

(1) {\em There exists a non-metrizable precompact abelian group with all dense subgroups separable}.

Recall that a non-trivial topological group is called {\em topologically simple} if it contains exactly two closed normal subgroups.
In \cite{HRT} and \cite{Peng}, it is shown that non-metrizable, topologically simple, precompact abelian groups exist.
Let $G$ be such a group and $H$ be its dense subgroup.
For each $0\neq h\in H$, $\hull{h}$ is a non-trivial subgroup of $G$.
The topological simplicity of $G$ implies that $\hull{h}$ is dense in $G$, so it must also be dense in $H$. Note that this also shows that $H$ is separable.

(2) {\em There exists  a non-metrizable locally compact abelian group with all dense subgroups separable}.

Let $\tau$ be an uncountable cardinal such that $\tau\leq 2^\omega$ and consider the group $K=\Z(2)^\tau$ with the usual product topology.
By the
Hewitt-Marczewski-Pondiczery Theorem \cite[Theorem 2.3.15]{Eng},
$K$ is separable.
Let $D$ be a countable dense subgroup of $K$.
Algebraically, $D$ is a direct summand of $K$, since it is a subspace of the linear space $K$ over the field $\Z(2)$.
Now write $K=D\oplus E$ for some subgroup $E$ of $K$.
% Note that $E$ is algebraically isomorphic to $K$.

Let us take $$G=\Z(4)^{(\omega)}\oplus E.$$

Then, $K=D\oplus E$ embeds into $G$ as a subgroup in the natural way.
 Indeed, it suffices to consider $D$ as the socle of $\Z(4)^{(\omega)}$, i.e., $D=\{2x: x\in \Z(4)^{(\omega)}\}$, and identify $E$ in $K$ with $E$ in $G$.
Note that $K$ has countable index in $G$.
 Altogether, we obtain

 $$2G=2(\Z(4)^{(\omega)})+2E=D+\{0\}=D.$$
 Let us equip $G$ with the group topology having $K$ as an open compact subgroup.
Note that this topology is Hausdorff and makes $G$ locally compact and non-metrizable.
It remains to check that every dense subgroup $H$ of $G$ is separable.

Since $K$ is open in $G$ and $H$ is dense, we have $G=K+H$.
Then, $$D=2G=2K+2H=\{0\}+2H=2H\leq H.$$

Take any countable subset $P$ of $H$ which contains a representative of every coset of $K$ in $G$.
Then, $H':=\hull{P}+D$ is a countable subgroup of $H$.
By the choice of $P$, for each $g\in G$, there exists $p_g\in P$ such that $p_g+K=g+K$.
So $H'\cap (g+K)$ contains $p_g+D$, which is dense in $g+K$.
In summary, $H'$ contains a dense subset of each coset of $K$ in $G$, so $H'$ is dense in $G$.
\end{example}

\section{Proof of Theorem B} \label{proof:ofthmB}

Let us recall the Levi-Mal'cev Structure Theorem (for connected compact groups).
For any topological group $G$, $Z(G)$ and $G'$ will be the centre and the (abstract) commutator subgroup, respectively.
And $G_0$ (resp. $Z_0(G)$) is the identity component of $G$ (resp. $Z(G)$).

Recall that a connected Lie group is called {\em simple} if it is non-abelian and has no proper closed connected normal subgroup but the trivial group. So if $N$ is a closed normal subgroup of a simple connected Lie group $G$, then $N$ must be discrete and hence contained in the centre of $G$.
%The commutator subgroup of $G_0$ is denoted by $G_0'$.
\begin{theorem}\label{ThLM}\cite[Theorem 9.24]{HM} Let $G$ be a connected compact group and $\Delta=Z_0(G)\cap G'$. Then $G=Z_0(G)G'$ and there exist a family $\{S_j:j\in J\}$ of simple simply connected compact Lie groups and a closed totally disconnected central subgroup $C$ of $G$ such that
\begin{itemize}
\item[(i)] $G'$ is a quotient group of $\prod_{j\in J}S_j$ (with the kernel contained in $\prod_{j\in J}Z(S_j)$); and
\item[(ii)] $G/C\cong (Z_0(G)/\Delta)\times \prod_{j\in J}(S_j/Z(S_j))$.
\end{itemize}
\end{theorem}
\begin{remark} Note that $C$ is indeed the centre of $G'$.
%Indeed, $C$ is the image of $\prod_{j\in J}{Z(S_j)}$ under the quotient mapping $\prod_{j\in J}S_j\to G'$ in (i) of the above Theorem.
%It now suffices to recall that epimorphisms between connected compact groups map the centre onto the centre of the image (see \cite[Theorem 9.28]{HM}).
It can be obtained from the quotient homomorphisms
$$\prod_{j\in J}{S_j}\to G'\to G'/C(\cong \prod_{j\in J}{S_j/Z(S_j)})$$
that $C$ is the image of $\prod_{j\in J}{Z(S_j)}$ under the first quotient mapping.
It now suffices to recall that epimorphisms between connected compact groups map the centre onto the centre of the image (see \cite[Theorem 9.28]{HM}).
 \end{remark}

The following lemma has been broadly accepted. However, I have been unable to find a specific reference to cite. As a result, I provide the proof in this manuscript to ensure its completeness and accuracy.

 %standard, and we only include its proof for completeness.
%\textcolor{blue}{For the sake of completeness, we summarize some results which are widely accepted knowledge. }
\begin{lemma}\label{Le3} Let $\kappa$ be an infinite cardinal and $\{H_\alpha:\alpha<\kappa\}$ be a family of non-trivial topological groups.
Then, the restricted product $H$ of $\{H_\alpha:\alpha<\kappa\}$ has density $\geq\kappa$.\end{lemma}
\begin{proof}
Let $P=\prod_{\alpha<\kappa}H_\alpha$.
Then $H$ is a dense subgroup of $P$.
Let $C$ be a subset of $H$ of cardinality less than $\kappa$.
Then, the cardinality of $supp(C):=\bigcup_{g\in C}supp(g)$ is less than $\kappa$ as well, since each $supp(g)$ is finite.
%\by Note that $supp(C)$ itself is less than $\kappa$. \ey
Let $\overline{C}$ be the closure of $C$ in $P$, then every $h=(h_\alpha)_{\alpha<\kappa}\in \overline{C}$ satisfies $supp(h)\subseteq supp(C)$.
Indeed, if $\beta\in supp(h)$, then $h_\beta$ is not the identity $1_\beta$ in $H_\beta$.
Therefore, $h$ has some
 open neighbourhood $U$ satisfying that $u_\beta\neq 1_\beta$ (i.e, $\beta\in supp(g)$) holds for each $u=(u_\alpha)_{\alpha<\kappa}\in U.$
We now proved that since $U\cap C$ is not empty.
Therefore, $|supp(\overline{C})|=|supp(C)|<\kappa$ holds, and hence $\overline{C}$ cannot be the whole $P$, i.e., $C$ is neither dense in $P$ nor in $H$.
\end{proof}

%The following is the last lemma before the proof of Theorem B.
%\begin{lemma}\label{Le2}Let $G$ be a connected locally compact group and $N$ a compact normal subgroup of $G$. Then $N=Z(N)N_0$.\end{lemma}
%\begin{proof}By \cite[Theorem 12.77]{HM2}, $N$ is contained in a connected compact subgroup $M$ of $G$.
%The totally disconnected normal subgroup $N/N_0$ of $M/N_0$ is central.
%Moreover, since continuous surjective homomorphisms between connected compact groups send the centre {\bf onto} the centre \cite[Theorem 9.28]{HM}, one has that
% $$N\subseteq Z(M)N_0.$$
 %Let $n\in N$ be arbitrary, then there exists $z\in Z(M)$ and $n_1\in N_0$ such that $n=zn_1$.
 %So $z=nn_1^{-1}\in N$.
 %That is, $z\in Z(M)\cap N\subseteq Z(N)$ and therefore, $n\in Z(N)N_0$.\end{proof}

\begin{proof}[\bf Proof of Theorem B] We follow the notation of Theorem \ref{ThLM}.
Let $\tau=w(G)$.
By contradiction, assume that every dense subgroup of $G$ has density $<\tau$.
As a consequence, the continuous homomorphic images of $G$, $Z_0(G)/\Delta$ and $\prod_{j\in J}(S_j/Z(S_j))$ also have dense subgroups only of density $<\tau$.
Applying Lemma \ref{Le3}, we deduce that $\kappa_1:=|J|<\tau$.
This implies that
\begin{equation}\label{e1} w(G')\leq \kappa_1\end{equation}
as $G'$ is a quotient group of $\prod_{j\in J}S_j$ and each factor $S_j$ is compact and metrizable.
On the other hand, $Z_0(G)/\Delta$ is abelian.
So Theorem A implies that
\begin{equation}\label{e2}\kappa_2:=w(Z_0(G)/\Delta)<\tau.\end{equation}
As a subgroup of $G'$, $\Delta$ satisfies
\begin{equation}\label{e3}w(\Delta)\leq \kappa_1\end{equation}
by (\ref{e1}).
Combining the inequalities (\ref{e2}) and (\ref{e3}), we obtain from Lemma \ref{3space} that:
\begin{equation}\label{e4}w(Z_0(G))\leq \kappa_1\cdot\kappa_2.\end{equation}
For the last step, recall that $G=Z_0(G)G'$.
So $G$ is a continuous image of $Z_0(G)\times G'$.
Then (\ref{e1}) together with (\ref{e4}) imply that
$$w(G)\leq\kappa_1\cdot\kappa_2<\tau,$$
which is a contradiction.
This completes our proof.
\end{proof}

\section{Limitations of Theorem B}

We now provide an example showing that, similarly to Theorem A, Theorem B admits no generalization to the locally compact case either.
We recall a lemma which is well-known.
For the readers' convenience, we provide the proof.
\begin{lemma}\label{Lenew}Let $G$ be a topological group and $H$ a dense subgroup.
Then for every normal subgroup $N$ of $H$, the closure of $N$ in $G$ is a normal subgroup of $G$.\end{lemma}
\begin{proof} Let $M$ be the closure of $N$ in $G$.
The mapping $\varphi: G\times M\to G, (g, m)\mapsto gmg^{-1}$ is continuous.
So $\varphi^{-1}(M)$ is closed in $G\times M$.
Note that $H\times N$ is contained in $\varphi^{-1}(M)$,
and so is its closure $G\times M$.
This implies that $M$ is normal in $G$.\end{proof}

Recall that a topological group is called {\em monothetic} if it contains a dense cyclic subgroup. Any monothetic group is necessarily abelian.
It is well-known that a compact abelian group is monothetic if and only if its Pontryagin dual group is algebraically a subgroup of the circle $\mathbb{T}$ \cite[Theorem 24.32]{HR}.

\begin{example} \label{this:example:feb82024} {\em For every uncountable cardinal $\tau\leq 2^\omega$, there exists a  connected locally compact group $G$ of weight $\tau$ whose dense subgroups are all separable.}

Let $K$ be the Pontryagin dual group of a (discrete) torsion-free abelian group $A$ of cardinality $\tau$. 
Then $A$ is a subgroup of the torus $\T$, hence $K$ is a connected (\cite[Theorem 24.25]{HR}) and compact monothetic (\cite[Theorem 24.32]{HR}) group, and it satisfies $w(K)=\tau$ (\cite[Theorem 24.14]{HR}).
Take any simple connected (necessarily noncompact) real Lie group $L$ with infinite centre. % (for example, the universal covering group of the (noncompact) symplectic group $\mathbf{Sp}(2, \R)$).
Consider $K$ and $L$ as (closed) subgroups of $H:=K\times L$ naturally embedded in $H$.
Let $\pi: H\to K$ be the projection.

\begin{claim}
For every dense subgroup $D$ of $H$, the intersection $D\cap L$ is dense in $L$.
\end{claim}
\begin{proof}
Let $C$ be the closure of $D\cap L$ in $H$.
By Lemma \ref{Lenew}, $C$ is a closed normal subgroup of $H$, and also of $L$.
Since $L$ is simple, $C$ is either discrete central or equal to the whole $L$.
 By contradiction, let us assume that $C$ is discrete and central.
In this case $C=D\cap L$.
Then $\pi(D)$ is algebraically isomorphic to $D/C$; so $D/C$ is abelian.
This implies that the group $H/C$, having $D/C$ as a dense subgroup, is also abelian.
In particular, $L/C$ is abelian.
%At the same time, $L$ and $L/C$ have isomorphic simple Lie algebras,
%which is a contradiction.
Let $p: L\to L/C$ be the canonical homomorphism. For $x,y\in L$, since $L/C$ is abelian, we have that 
$$p(x)p(y)p(x)^{-1}p(y)^{-1}=p(xyx^{-1}y^{-1})$$
is the identity in $L/C$.
It follows that $xyx^{-1}y^{-1}\in\ker p= C$.
Thus we obtain the continuous mapping
$$L\times L\to C, (x,y)\mapsto xyx^{-1}y^{-1}.$$
Since $L$ is connected and $C$ is discrete, it then follows that the image of this mapping is the singleton $\{1_L\}$, where $1_L$ is the identity of $L$.
In other words, $x$ and $y$ commutes for any pair of $x,y$ in $L$. So $L$ is abelian, contradicting to that $L$ is simple.
So the equality $C=L$ must hold, thereby proving that $D\cap L$ is dense in $L$.
\end{proof}

Let $a\in K$ be such that $\hull{a}$ is dense in $K$, and let $b\in Z(L)$ be of infinite order.
Then $P:=\hull{(a,b)}\in H$ is a discrete cyclic central subgroup of $H$.
Let $G=H/P$ and $E$ be a dense subgroup of $G$.
The preimage $D=p^{-1}(E)$ under the canonical mapping $p: H\to H/P=G$ is a dense subgroup of $H$ containing $P$.
So $D\cap L$ is dense in $L$ by the above claim.
Since $L$ is second countable, $D\cap L$ contains a countable dense subgroup $M$.
The subgroup $PM$ of $D$ is then countable.
Let us show that $PM$ is dense in $H$.
On one hand, the closure $T$ of $PM$ in $H$ contains $L$, since $T\cap L\supseteq M$ and $M$ is dense in $L$.
This implies that $T=Q\times L$, where $Q$ is the closure of $\pi(T)$ or, equivalently, of 
$$\pi(PM)=\pi(P)=\hull{a}.$$
%$$\pi(T)\supseteq\pi(PM)=\pi(P)=\hull{a}$$ in $K$.
By the denseness of $\hull{a}$ in $K$, one obtains that $Q=K$, and hence $PM$ is dense in $H$.
Now $PM/P$ is a countable dense subgroup of $E$; so $E$ is separable.
It remains to show that $w(G)=\tau$.
 Note that this follows from the equality $$w(G)\cdot w(P)=w(H)=\tau$$
as established by Lemma \ref{3space} and the fact that $w(P)=\omega$.
\end{example}
%Let us see that weight of the locally compact group $G$ in the above example is restricted, as
%things may become different if  $G$ is ``large enough''.
 Let us show that in Example 5.2, the weight of the
 locally compact group $G$ is bounded by $2^\omega$, while the
 things may become different if the groups are “large
 enough”.
In proposition \ref{this:prop:feb82024} below, we will demonstrate that the weight of a locally compact group $G$, as in the preceding example, can never exceed $2^\omega$.
The celebrated Gleason-Yamabe Theorem \cite{Ya1,Ya2} states
that every connected locally compact group is a projective limit of Lie groups.
Thus Iwasawa's connected (L)-groups \cite{Iwa} are exactly connected locally compact groups and therefore his structure theorems for connected (L)-groups can be applied to the connected locally compact groups.
%One of Iwasawa's main theorems in \cite{Iwa}, which is called
The {\em Iwasawa Local Splitting Theorem} (\cite{Iwa}) asserts that every connected locally compact group is locally isomorphic to the direct product of a compact  group and a Lie group.
In Chapter 13 of Hofmann and Morris' book \cite{HM2}, a global version of this theorem is given. More precisely, the local isomorphism is shown to be given by a quotient homomorphism with discrete kernel.
So, the structure theory of connected locally compact groups is largely reduced to Lie groups and compact groups.

Let us recall the following version of the Local Splitting Theorem of Iwasawa\footnote{In \cite{HM2}, Hofmann and Morris give a more general version for pro-Lie groups, i.e., inverse limits of Lie groups.}:
\begin{theorem}[{\cite[Theorem 13.17]{HM2}}]\label{ThLS}Let $G$ be a connected locally compact group. Then every identity neighbourhood contains a compact normal subgroup $N$ of $G$ such that $G/N$ is a Lie group and there exists a continuous homomorphism $\varphi: \widetilde{G/N}\to G$ such that $\mu: N\times \widetilde{G/N}\to G$ defined by $(y, x)\mapsto y\varphi(x)$ is an open (hence, surjective) continuous homomorphism with a discrete kernel, where $\widetilde{G/N}$ is the universal covering group of $G/N$.\end{theorem}

\begin{proposition} \label{this:prop:feb82024} Let $G$ be a connected locally compact group with $w(G)>2^\omega$.
Then $G$ admits a dense subgroup of density $w(G)$.\end{proposition}
\begin{proof} Take any identity neighbourhood in $G$, we then find such a compact normal subgroup $N$ as described in Theorem \ref{ThLS}. 
We also use the notation $\varphi$, $\mu$ and $\widetilde{G/N}$ same as in the theorem.
Let $L$ be the closure of $\varphi(\widetilde{G/N})$ in $G$ and set $K=L\cap N$.
Note that $\{1\}\times \widetilde{G/N}$ is normal in $N\times \widetilde{G/N}$, so $L$ a normal subgroup of $G$.
Since $L$ contains a dense subgroup which is a continuous image of a connected Lie group, it is separable.
So one has $w(L)\leq 2^\omega,$ by \cite[Theorem 1.5.7]{Eng}.
We now have  that $w(G/L)=\tau$ by Lemma \ref{3space}, where $\tau=w(G)$.
Note that the mapping $\mu$ defined in Theorem \ref{ThLS} is surjective, so $G=NL$ and $G/L$ is topologically isomorphic to $N/K$.
It follows that $G/L$ is connected and compact.
By Theorem B, $G/L$ has a dense subgroup of density $\tau$.
The preimage of this subgroup under $G\to G/L$ (which is a dense subgroup of $G$) has density $\geq \tau$.
It suffices to note that the strict inequality never holds due to the equality $w(G)=\tau$.
\end{proof}

\begin{remark} A standard argument can be applied to show that a connected locally compact abelian group $G$ admits a dense subgroup $H$ with $d(H)=w(G)$. \end{remark}
\section{Theorem C and its consequences} \label{proof:ofthmC}

We shall need the so-called {\em Countable Layer Theorem} for compact groups of K. Hofmann and S. Morris.
Recall that for a topological group $G$, a {\em (topological) characteristic} subgroup is a subgroup invariant under all topological automorphisms of $G$. So if $H$ is a topological group containing $G$ as a normal subgroup, then every characteristic subgroup of $G$ is normal in $H$.
\begin{theorem}[\cite{HM1}] \label{ThCL} Let $G$ be a compact group. Then there exists a chain of closed characteristic subgroups (that possibly stops after finitely many steps)
$$G=\Omega_0 G\geq \Omega_1 G\geq \Omega_2G\geq...\geq \bigcap_{n=0}^\infty\Omega_n G=Z_0(G_0)$$
of $G$ such that each factor $\Omega_{n}G/\Omega_{n+1}G$ is a product of simple compact groups\footnote{A simple compact group is either a simple connected compact Lie group with trivial centre or a finite (possibly abelian) simple group.}.
\end{theorem}
Here we recall that $G_0$ is the identity component of $G$ and $Z_0(G_0)$ is the identity component of $Z(G_0)$, the centre of $G_0$. 
\begin{proof}[\bf Proof of Theorem C]
Let $\tau$ be the supremum of densities of subgroups of $G$.
Then $\tau\leq w(G)$.
\begin{claim} \label{this:claim:feb82020}
The equality $\tau=w(G)$ holds if $G$ is compact.
\end{claim}
\begin{proof}
Let $G=\Omega_0 G\geq \Omega_1 G\geq \Omega_2G\geq...\geq \Omega_nG\geq...$ be a chain of closed subgroups of $G$ as in Theorem \ref{ThCL}.
By our assumption, for each $n\in \omega$, all subgroups of $\Omega_nG$ are of density $\leq \tau$.
It follows that any subgroup of $\Omega_{n}G/\Omega_{n+1}G$ has density not greater than $\tau$ as well.
By Lemma \ref{Le3}, this implies that $\Omega_{n}G/\Omega_{n+1}G$ is the product of at most $\tau$ many simple compact groups.
So we have
$$\kappa_n:=w(\Omega_{n}G/\Omega_{n+1}G)\leq \tau.$$
Proceeding by induction, by Lemma \ref{3space}, it follows that $w(G/\Omega_nG)\leq \tau$ for all $n\in \omega$.
Note that $G/Z_0(G_0)$ is the limit of the inverse system $\{G/\Omega_nG; \varphi^j_{i}\}$ indexed by $\omega$, where for each pair $i<j$, $\varphi^j_i: G/\Omega_jG\to G/\Omega_iG$ is the canonical quotient mapping.
So $G/Z_0(G_0)$ embeds into the product $\prod_{n\in \omega}G/\Omega_nG$; thus
$$w(G/Z_0(G_0))=\sup_{n\in \omega}\kappa_n \leq \tau.$$
To conclude this case, by Lemma \ref{3space}, it suffices to note that the inequality
$$w(Z_0(G_0))\leq \tau$$
follows from Theorem A.
\end{proof}
%Let $G=\Omega_0G\geq \Omega_1 G\geq...$ be the chain of closed subgroups of $G$ as in Theorem \ref{ThCL}.
Let us now assume that $G$ is locally compact.
%Denote by $G_0$ the identity component of $G$. 
Then, $G/G_0$ is totally disconnected.
The well-known van Dantzig Theorem yields that $G/G_0$ has an open compact subgroup $U$.
Let $V$ be the preimage of $U$ in $G$ under the canonical homomorphism of $G$ onto $G/G_0$. Then $V$ is an open subgroup of $G$ containing $G_0$ and $V/G_0$ is compact.
The solution of Hilbert's fifth problem then yields that $V$ contains a compact normal subgroup $N$ such that $V/N$ is a Lie group.
%In particular, $$w(V/N)=\omega\leq \tau.$$
%By the Claim \ref{this:claim:feb82020}, 
By the above argument and the equality $w(V/N)=\omega$,
we have $w(N)\leq \tau$.
This implies that $w(V)\leq \tau$.
We also have $|G/V|\leq \tau$ because $V$ is open in $G$ and $d(G)\leq \tau$ holds.
Thus, 
$$w(G)=w(V)\cdot |G/V|\leq \tau.$$

Now we consider the second part of the theorem and assume that the cofinality of $\lambda:=w(G)$ is uncountable.
By the above argument and Lemma \ref{3space}, we have
$$w(G)=w(V)\cdot |G/V|=w(N)\cdot w(V/N)\cdot |G/V|=w(N)\cdot |G/V|.$$
If $w(N)<\lambda$, then $|G/V|=\lambda$.
Therefore,  $G$ is of density $\lambda$ because $V$ is open in $G$.

Now let us consider the case $w(N)=\lambda$.
Using the Countable Layer Theorem again, one obtains the following sequence of closed normal subgroups of $N$:

$$N=\Omega_0N\geq \Omega_1N \geq \Omega_2N\geq...\geq \bigcap_{n=1}^\infty\Omega_nN=Z_0(N_0),$$
such that $\Omega_nN/\Omega_{n+1}N$ is a direct product of simple compact groups.
If the abelian group $Z_0(N_0)$ has weight $\lambda$, then it follows directly from Theorem A that $N$, hence also $G$, has a subgroup of density $\lambda$.
If $w(Z_0(N_0))$ is less than $\lambda$, then  there exists $n\in \omega$ such that $$w(\Omega_nN/\Omega_{n+1}N)=w(N)=\lambda,$$
since $\lambda$ is not cofinal with $\omega$.
According to Lemma \ref{Le3}, $\Omega_nN/\Omega_{n+1}N$ has a dense subgroup of density $\lambda$, and so does $\Omega_nN$.
Therefore, $G$ contains a subgroup of density $\lambda$.
\end{proof}

Recall that a topological space is called {\em hereditarily separable} if all subspaces of it are separable. %A {\em dyadic compactum} is a continuous image of $D^\tau$ for some cardinal $\tau$, where $D=\{0,1\}$.
%A topological space $X$ is called {\em nondegenerate} if it cannot be represented as the union of countably many closed subspaces each is of weight strictly less than $w(X)$.

To close this section, we point out that Theorem C is a generalization of the following:

\begin{fact}
Every hereditarily separable compact group is metrizable.
\end{fact}
Since every compact group is dyadic and the tightness of every hereditarily separable space is countable,
the above Fact is a very special case of the following impressive theorem proved by Arhangel'skii and Ponomarev in \cite{AP} in the year 1968:
{\em Every dyadic compact space of countable tightness is metrizable.}
We refer the readers to \cite{AP} or \cite{AT} for the notions.

%\begin{proof}
%By the results of Hagler \cite{Hag}, Gerlits \cite{Ger} and Efimov \cite{Efi},
%every infinite dyadic compactum $X$ of weight $\tau$ contains $D^\tau$ as a subspace if $X$ is nondegenerate.
%According to the Ivanovskij-Kuz'minov Theorem \cite[Theorem 4.1.7]{AT}, every compact group is a dyadic compactum.
%Moreover, since the weight and the local weight coincide in a compact group, it is a direct
% consequence of Baire Category Theorem that every compact group is nondegenerate.
%We then obtain the contrapositive: namely, any non-metrizable compact group $G$ of weight $\tau$ contains a copy of $D^\tau$.
%Since $\tau>\omega$, the space $D^\tau$ is not hereditarily separable, and so neither is $G$.
%\end{proof}

\section{Transfinite nilpotent profinite groups} \label{transfinite:section}
Thereafter, the identity element of a group will be noted by $1$.
For a compact group $G$, let $G_0=G$\footnote{note that $G_0$ was used to denote the identity component in previous sections} and for any ordinal $\alpha>0$, set
$$G_\alpha=\begin{cases}
           \overline{[G, G_\beta]} &\mbox{if~$\alpha=\beta+1$;}\\
           \bigcap_{\beta<\alpha}G_\alpha &\mbox{if~$\alpha$~is~limit.}
\end{cases}$$
A compact group $G$ is called \emph{transfinite nilpotent} if $G_\alpha=\{1\}$ holds for some ordinal $\alpha$.
It is known that every transfinite nilpotent compact group $G$ is \emph{countably nilpotent}, that is, $\bigcap_{n<\omega}G_n=\{1\}$ holds (see \cite[Corollary 1.10]{HM1}).

Using Theorem A, we are going to prove the following:
\begin{corollary}\label{Coro:Jan29}Let $G$ be a transfinite nilpotent profinite group. If the dense subgroups of $G$ are all separable, then $G$ is metrizable.\end{corollary}
A \emph{profinite} group is a totally disconnected compact group. In such a group $G$, we call a subset $X$ a \emph{set of generators converging to 1} if $\hull{X}$ is dense in $G$ and every neighbourhood of 1 contains all but finite many elements of $X$. So a countably infinite set of generators converging to 1 is a sequence converging to 1 which generates a dense subgroup.
It is known that a profinite group is metrizable if and only if it contains a countable set of generators converging to 1 (see \cite[Corollary 2.6.3]{RZ}). If such a set is infinite, then it is also a sequence converging to 1; if it is finite, one can easily extend it into a sequence (with pairwise distinct elements) converging to 1 as well, provided the group is non-discrete, equivalently, infinite. 
In summary, an infinite metrizable profinite group always has a sequence converging to 1 which generates a dense subgroup.

To prove Corollary \ref{Coro:Jan29}, we also need a useful theorem of Varopoulos.
\begin{theorem}[{\cite[The Sequence Lifting Theorem]{Var}}] \label{sequence:lift:thm} Let $G$ be a locally compact group, $H$ a closed subgroup of $g$, and $p:G\to G/H$ the natural projection onto the left coset space $G/H$.
Then any convergent sequence $y_n\to y_0 ~(n\to \infty)$ in $G/H$ can be ``lifted'' to a convergent sequence $x_n\to x_0 ~(n\to \infty)$ in $G$ such that $p(x_n)=y_n$ for all $n\in \omega$.\end{theorem}
\begin{proof}[{\bf Proof of Corollary \ref{Coro:Jan29}}]
Since $G$ is transfinite nilpotent, it is also countably nilpotent.
For $n\in \omega$, we let $G_n$ be the closed characteristic subgroup defined at the beginning of this section.
If $G$ is not metrizable, then there exists at least one integer $m\geq 0$ such that $K_m:=G_m/G_{m+1}$ is not metrizable.
Take $n$ be the least integer having this property.
Let $H=G/G_{n+1}$. Then $K_n$ is in the centre of $H$.
 Denote by $\pi$ the quotient homomorphism $H\to H/K_n$.
Take a minimal closed subgroup $L$ of $H$ such that $\pi(L)=H/K_n$\footnote{see, for example, \cite[Lemma 2.8.15]{RZ} for the existence of such a subgroup $L$}.
\begin{claim}
$L$ is metrizable.
\end{claim}
\begin{proof}
By the choice of $n$, the quotient $H/K_n\cong G/G_n$ is necessarily metrizable.
 We may then take a sequence $(y_i)_{i\in \mathbb{N}}$ converging to 1 in $H/K_n$ such that $\{y_i:i\in \mathbb{N}\}$ generates a dense subgroup of $H/K_n$.
By Theorem \ref{sequence:lift:thm},  there exists a convergent sequence $(x_i)_{i\in \mathbb{N}}$ in $L$ such that $\pi(x_i)=y_i$ for each $i\in \mathbb{N}$.
Without loss of generality, we may assume that $\lim\limits_{i\to \infty}x_i=1$.
Denote by $L'$ the closed subgroup of $L$ generated by $\{x_i:i\in \mathbb{N}\}$.
Then, $\pi(L')$ is a closed subgroup of $H/K_n$ containing $\{y_i:i\in \mathbb{N}\}$; so $\pi(L')=H/K_n$.
By the minimality of $L$, we conclude that $L'=L$.
Since $\{x_i: i\in \mathbb{N}\}$ is a generating set of the profinite group $L$ converging to 1, $L$ is metrizable. \end{proof}

By the equality $\pi(L)=H/K_n$, one obtains that $H=LK_n$.
Since $K_n$ is in the centre of $H$, it normalizes $L$. This implies that $L$ is a normal subgroup of $H$.
Define $\Delta=L\cap K_n$.
Then $\Delta$, as a subgroup of $L$, is metrizable and hence $K_n/\Delta$ is non-metrizable.
So it has a non-separable dense subgroup, according to Theorem A.
 Since $$H/L\cong K_n/\Delta,$$
it follows that $H/L$ also has such a dense subgroup $D$.
Then the preimage of $D$ under the canonical homomorphism $H\to H/L$ provides a dense subgroup of $H$ which is not separable.
Similarly, this implies that $G$ also has a non-separable dense subgroup because $H=G/G_{n+1}$ is a quotient group of $G$.
\end{proof}

\begin{remark}
A topological group is called \emph{pronilpotent} if it is the  limit of an inverse system of finite discrete nilpotent groups \cite{RZ}, or equivalently, if it can be embedded into the direct product of finite discrete nilpotent groups as a closed subgroup.
On one hand, a product of finite nilpotent groups is countably nilpotent, and so are its subgroups. On the other hand, if $G$ is a profinite countably nilpotent group, then $G/G_n$ is nilpotent (and hence pronilpotent) and $G$ embeds into $\prod_{n\in \mathbb{N}}G/G_n$; thus $G$  itself is pronilpotent. Here, $G_n$ is again be the subgroup of $G$ defined at the beginning of this section. In summary, for profinite groups we have
$$\mbox{transfinite~nilpotent}=\mbox{countably~nilpotent}=\mbox{pronilpotent}.$$
\end{remark}
According to the remark above, one has the following.
\begin{proposition}Let $G$ be a profinite group with all dense subgroup separable. If $G$ is also pronilpotent, then $G$ is metrizable.\end{proposition}
For any prime number $p$, it is widely known that finite $p$-groups are nilpotent.
A standard argument then shows that pro-$p$ groups (i.e., projective limits of finite $p$-groups) are pronilpotent \cite{RZ}. This implies the following:
\begin{corollary} A pro-$p$ group with all dense subgroups separable is metrizable.\end{corollary}

\section*{Acknowledgements}
This paper has been improved a lot by a detailed report kindly provided by the referee. The author wants to express the acknowledgement.
I also would like to thank Professor Wei He for his numerous suggestions and thank V\'{\i}ctor Hugo Ya\~{n}ez for helping me improving the earlier version of this paper.

\end{document}